\documentclass[12pt,reqno]{amsart}
\usepackage{jrmacros}

\usepackage{mathtools}
\mathtoolsset{showonlyrefs}

\usepackage{appendix}

\theoremstyle{plain}
\newtheorem{theorem}{Theorem}[section]
\newtheorem{proposition}[theorem]{Proposition}

\newtheorem*{theorem*}{Theorem}
\newtheorem*{proposition*}{Proposition}
\newtheorem*{corollary*}{Corollary}
\newtheorem*{lemma*}{Lemma}
\newtheorem*{conjecture*}{Conjecture}

\theoremstyle{definition}

\newtheorem{question}{Problem}

\newtheorem*{definition*}{Definition}
\newtheorem*{example*}{Example}
\newtheorem*{question*}{Question}
\newtheorem*{philosophy*}{Philosophy}

\theoremstyle{remark}
\newtheorem{remark}[theorem]{Remark}

\newtheorem*{remark*}{Remark}

\def\p{\mathbf{p}}

\def\p{\mathbf{p}}

\begin{document}
\title{A general family of congruences for Bernoulli numbers}
\author{Julian Rosen}
\date{\today}
\maketitle

\begin{abstract}
We prove a general family of congruences for Bernoulli numbers whose index is a polynomial function of a prime, modulo a power of that prime. Our family generalizes many known results, including the von Staudt--Clausen theorem and Kummer's congruence.
\end{abstract}

\section{Introduction}

\subsection{Bernoulli numbers}

The Bernoulli numbers $B_n$ are rational numbers defined by the power series expansion
\[
\frac{x}{e^x-1}=\sum_{n\geq 0}B_n\frac{x^n}{n!}.
\]
The first few values are:
\renewcommand{\arraystretch}{1.2}
\[
\begin{array}{|c||c|c|c|c|c|c|c|c|c|c|c|c|c|}
\hline n&0&1&2&3&4&5&6&7&8&9&10&11&12\\\hline
B_n&1&-\frac{1}{2}&\frac{1}{6}&0&-\frac{1}{30}&0&\frac{1}{42}&0&-\frac{1}{30}&0&\frac{5}{66}&0&-\frac{691}{2730}\\\hline
\end{array}.
\]
Terms of odd index greater than $1$ vanish, and the non-zero terms alternate in sign.

The Bernoulli numbers are known to have interesting arithmetic properties. One well-known example is the von Staudt--Clausen Theorem, which says that for every even integer $n\geq 2$, the quantity
\[
B_n+\sum_{\substack{p\text{ prime,}\\p-1\mid n}}\frac{1}{p}
\]
is an integer. In particular, the denominator of $B_n$ is the product of those primes $p$ such that $p-1$ divides $n$.
Another well-known result is Kummer's congruence, which says that for all non-negative even integers $m$, $n$ not divisible by $p-1$, satisfying $m\equiv n\mod \varphi(p^k)$,
\begin{equation}
\label{eqkum}
\lp 1-p^{m-1}\rp\frac{B_m}{m}\equiv \lp 1-p^{n-1}\rp\frac{B_n}{n}\mod p^k.
\end{equation}
Here, a congruence between rational numbers modulo $p^k$ means the $p$-adic valuation of their difference is at least $k$ (that is, congruence modulo $p^k\Z_{(p)}$). Many generalizations of \eqref{eqkum} are known, some involving additional terms, some relaxing the restriction that $p-1\nmid m$. Several generalizations are given in \cite{Coh07} (\S9.5, \S11.4.2)

The results of von Staudt--Clausen and Kummer can be expressed in terms Bernoulli numbers whose index is a polynomial in $p$.
\begin{enumerate}
\item Let $f(t)\in\Z[t]$ have positive leading coefficient, and set $\delta=-1$ if $f(1)=0$, and $\delta=0$ if $f(1)\neq 0$. Then the von Staudt--Clausen theorem implies that
\begin{equation}
\label{vscpoly}
pB_{f(p)}\equiv\delta\mod p
\end{equation}
for every prime $p>|f(1)|$.
\item Let $f(t)$, $g(t)\in\Z[t]$ be distinct non-constant polynomials with positive leading coefficient, and suppose that $f(1)=g(1)\neq 0$. Let $k$ be the largest power of $t$ dividing $f(t)-g(t)$. Then Kummer's congruence implies that
\begin{equation}
\label{kpoly}
\frac{B_{f(p)}}{f(p)}\equiv \frac{B_{g(p)}}{g(p)}\mod p^{k+1}
\end{equation}
holds for every prime $p>|f(1)|$.

\end{enumerate}
\smallskip

\noindent Here we consider the following problem.

\begin{question}
Given polynomials $f_1(t),\ldots,f_n(t)\in\Z[t]$ with positive leading coefficients, rational functions $g_0(t),g_1(t),\ldots,g_n(t)\in\Q(t)$, and a non-negative integer $N$, determine under which circumstances the congruence
\begin{equation}
\label{bsc}
\sum_{i=1}^n g_i(p)B_{f_i(p)}\equiv g_0(p)\mod p^N
\end{equation}
holds for every sufficiently large prime $p$.
\end{question}

The polynomial form of the von Staudt--Clausen theorem \eqref{vscpoly} and Kummer's congruence \eqref{kpoly} are examples of \eqref{bsc}. Other examples are known. For instance, it is a result of Z.-H.\ Sun \cite{Sun00} that for integers $k$, $b$, with $b\not\equiv 0\mod p-1$ even and $k\geq 0$:
\begin{equation}
\label{s1}
\frac{B_{k(p-1)+b}}{k(p-1)+b}\equiv k\frac{B_{p-1+b}}{p-1+b}-(k-1)(1-p^{b-1})\frac{B_b}{b}\mod p^2,
\end{equation}
\begin{gather}
\frac{B_{k(p-1)+b}}{k(p-1)+b}\equiv {k\choose 2}\frac{B_{2(p-1)+b}}{2(p-1)+b}-k(k-2)\frac{B_{p-1+b}}{p-1+b}\\
\hspace{40mm}+{k-1\choose 2}(1-p^{b-1})\frac{B_b}{b}\mod p^3.\label{s2}
\end{gather}

\subsection{Results}
The main result of this paper is a criterion for \eqref{bsc} to hold.

\begin{theorem}
\label{thmain}
Fix an integer $N$, non-constant polynomials $f_1(t),\ldots,f_n(t)\in\Z[t]$ with positive leading coefficients, and rational functions $g_0(t),\ldots,g_n(t)\in\Q(t)$. Write $v_t$ for the $t$-adic valuation on $\Q(t)$, and set
\[
M=\min\{v_t(g_i(t))\}.
\]
Then the congruence
\begin{equation}
\label{bscth}
\sum_{i=1}^n g_i(p)B_{f_i(p)}\equiv g_0(p)\mod p^N
\end{equation}
holds for every sufficiently large prime $p$ if all of the following conditions hold.
\begin{enumerate}
\item \[
v_t\lp g_0(t)-\lp 1-\frac{1}{t}\rp\sum_{\substack{i\\f_i(1)=0}}g_i(t)
-\sum_{\substack{i\\f_i(1)\geq 2}}\lp 1-t^{f_i(1)-1}\rp \frac{B_{f_i(1)}}{f_i(1)}g_i(t)f_i(t)
\rp\geq N.
\]
\item For every even, non-positive integer $k\in\{f_i(1)\}$ and every $1\leq m\leq N-M$,
\[
v_t\lp\sum_{\substack{i\\f_i(1)= k}} g_i(t)f_i(t)^m\rp\geq N+1-m.
\]
\item For every even, positive integer $k\in\{f_i(1)\}$ and $2\leq m\leq N-M$,
\[
v_t\lp\sum_{\substack{i\\f_i(1)= k}} g_i(t)\big(f_i(t)^m-k^{m-1}f_i(t)\big)\rp\geq N+1-m.
\]
\end{enumerate}
\end{theorem}
The verification of the conditions (1)--(3) of Theorem \ref{thmain} is a finite computation. The condition that the polynomials $f_i(t)$ are non-constant is for simplicity and is inessential: for each $f_i(t)$ that is constant, the term $g_i(p)B_{f_i(p)}$ can be moved to the right hand side of \eqref{bscth} and made part of $g_0(p)$.

\begin{remark}
It can be checked that Theorem \ref{thmain} implies both \eqref{vscpoly} and \eqref{kpoly}, so we may view Theorem \ref{thmain} as a common generalization of the von Staudt--Clausen Theorem and Kummer's congruence. The congruences \eqref{s1} and \eqref{s2} also follow from Theorem \ref{thmain}, as do many of the congruences for Bernoulli numbers given in \cite{Coh07}.
\end{remark}

The author does not know whether the converse of Theorem \ref{thmain} holds. If it does, this is likely quite difficult to prove. In particular, the converse of Theorem \ref{thmain} would imply that there exist infinitely many non-Wolstenholme primes (that is, infinitely many primes $p$ for which $p\nmid B_{p-3}$), which is an open question.

\section{The Kubota-Leopoldt $p$-adic zeta function}

The Riemann zeta function $\zeta(s)$ takes rational values at the non-positive integers, and we have a formula
\[
\zeta(-n)=-\frac{B_{n+1}}{n+1}.
\]
Fix a prime $p\geq 3$. Kummer's congruence implies that the values
\[
(1-p^{-s})\zeta(s)
\]
are $p$-adically uniformly continuous if $s$ is restricted to the negative integers in a fixed residue class modulo $p-1$. For $n\geq 2$ an integer, the \emph{Kubota-Leopoldt $p$-adic zeta value} is defined by
\[
\zeta_p(n):=\lim_{\substack{s\to n\text{ $p$-adically},\\s \leq 0\\s\equiv n\mod p-1}}(1-p^{-s})\zeta(s).
\]
The function $n\mapsto\zeta_p(n)$ is not $p$-adic analytic, but comes from $p-1$ analytic functions, one for each residue class modulo $p-1$.

The following proposition gives a bound on the valuation of the power series coefficients for the $p$-adic zeta function. It is essentially a repackaging of known results.
\begin{proposition}
\label{propser}
Let $p$ be an odd prime, $k$ an even residue class modulo $p-1$. Then there exist coefficients $a_i(p,k)\in\Q_p$ for $i=0,1,2,\ldots$ such that for every non-negative integer $n$ with $n\equiv k\mod p-1$, there is a convergent $p$-adic series identity
\begin{equation}
\label{B}
B_n=(1-p^{n-1})^{-1}\sum_{i\geq 0}a_i(p,k)n^i.
\end{equation}
The coefficients $a_i(p,k)$ satisfy the following conditions:
\begin{enumerate}
\item
\[
a_{0}(p,k)=\begin{cases}1-\frac{1}{p}&\text{ if }k\equiv 0\mod p-1,\\
0&\text{ otherwise,}\end{cases}
\]
\item for all $i$, $p$, $k$,
\[
v_p(a_{i}(p,k))\geq \frac{p-2}{p-1}i-2,
\]
\item for $p\geq i+2$ and all $k$,
\[
v_p(a_{i}(p,k))\geq i-1.
\]
\end{enumerate}
\end{proposition}
\begin{proof}
Let $\omega_p$ be the Teichm\"{u}ller character. For every residue class $k$ modulo $p-1$, there is a Laurent series expansion for the Kubota-Leopoldt $p$-adic $L$-function:
\[
L_p(s,\omega_p^k)=\sum_{i\geq -1}c_{i,k} (s-1)^i.
\]
The relationship between $L_p$ and Bernoulli numbers is given by
\[
-(1-p^{n-1})\frac{B_n}{n}=L_p(1-n,\omega_p^n),
\]
so we may take $a_i(p,k)=(-1)^ic_{i-1,k}$. Statement (1) of the Proposition follows form the fact that $L_p(s,\omega_p^k)$ has a simple pole of residue $1-\frac{1}{p}$ at $s=1$ if $k\equiv 0\mod p-1$, and is analytic otherwise.

To get the desired bounds on the valuation of the $a_i$, we use \cite{Was82}, Theorem 5.11 (in the case $\chi=\omega_p^k$, $f=p$, $F=p(p-1)$):
\begin{equation}
\label{firstL}
L_p(s,\omega_p^k)=\frac{1}{p(p-1)}(s-1)^{-1}\sum_{j=0}^{\infty} p^j(p-1)^j{1-s\choose j} B_j\sum_{\substack{a=1\\p\nmid a}}^{p(p-1)}\frac{\omega_p^k(a)}{a^j}\langle a\rangle^{1-s},
\end{equation}
where $\langle a\rangle := a/\omega_p(a)\equiv 1\mod p$. Write $\langle a\rangle=1+p q_a$, with $q_a\in\Z_p$. We can expand $\langle a\rangle^{s-1}$ as a binomial series to obtain
\[
L_p(s,\omega_p^k)=\frac{1}{p(p-1)}(s-1)^{-1}\sum_{i,j=0}^{\infty} p^{i+j}(p-1)^j{1-s\choose i}{1-s\choose j} B_j   \sum_{\substack{a=1\\p\nmid a}}^{p(p-1)}\frac{\omega_p^k(a)}{a^j} q_a^i.
\]
The innermost summation is $p$-integral, so we conclude that
\[
a_i(p,k)\in \frac{1}{p} \sum_{s\geq t \geq i} B_t\frac{p^s}{s!}\Z_p. 
\]
Now we have $v_p(B_t)\geq -1$ by the von Staudt--Claussen Theorem, and $v_p(s!)\leq s/(p-1)$, so
\[
v_p(a_i(p,k))\geq \frac{p-2}{p-1}i - 2,
\]
which implies statement (2) of the Proposition. Finally, if $p\geq i+2$, then
\[
\frac{p^s}{s!}B_t\in p^i\Z_p
\]
for all $s\geq t\geq i$, so in this case we have $v_p(a_i(p,k))\geq i-1$. This completes the proof.
\end{proof}

\section{Proof of the Theorem}
We start with an easy fact.
\begin{proposition}
Suppose $g(t)\in\Q(t)$ is non-zero. Then
\begin{equation}
\label{eqval}
v_p(g(p))=v_t(g(t))
\end{equation}
for all but finitely many primes $p$.
\end{proposition}
\begin{proof}
We can write
\[
g(t)=t^{v_t(g(t))}\frac{a(t)}{b(t)},
\]
where $a(t)$, $b(t)\in\Z[t]$ satisfy $a(0)b(0)\neq 0$. The equality \eqref{eqval} then holds for every prime not dividing $a(0)b(0)$.
\end{proof}

We are now ready to prove Theorem \ref{thmain}.
\begin{proof}[Proof of Theorem \ref{thmain}]
Suppose $f_1,\ldots,f_n$ are non-constant integer polynomials with positive leading coefficient, $g_0,\ldots,g_n$ are rational functions, and $N$ is an integer, satisfying conditions (1)--(3). We would like to prove that the quantity
\[
A(p):=g_0(p) -\sum_{i=1}^n g_i(p)B_{f_i(p)}
\]
is divisible by $p^N$ for every sufficiently large prime $p$. We can compute an expression for $A(p)$ in terms of the $a_i(p,k)$ from the previous section:
\begin{align*}
A(p)&=g_0(p)-\sum_{k\in\Z/(p-1)}\sum_{\substack{i\\f_i(p)\equiv k\, (p-1)}}g_i(p) B_{f_i(p)}\\
&=g_0(p)-\sum_{\substack{k\in\Z/(p-1)\\m\geq 0}}\lp\sum_{\substack{i\\f_i(p)\equiv k\, (p-1)}} \lp1-p^{f_i(p)-1}\rp^{-1} g_i(p)f_i(p)^m\rp a_m(p,k).
\end{align*}
We have $f_i(p)\equiv f_i(1)$ modulo $p-1$, so for every $p>\max\{f_i(1)-f_j(1)\}$,
\[
A(p)=g_0(p)-\sum_{\substack{k\in\Z\\m\geq 0}}\lp\sum_{\substack{i\\f_i(1)= k}} \lp1-p^{f_i(p)-1}\rp^{-1} g_i(p)f_i(p)^m\rp a_m(p,k).
\]
Because $f_i(t)$ is non-constant, we have $v_p(p^{f_i(p)})\to\infty$ as $p\to\infty$, so 
\begin{equation}
\label{eq1}
A(p)\equiv g_0(p)-\sum_{\substack{k\in\Z\\m\geq 0}}\lp\sum_{\substack{i\\f_i(1)= k}} g_i(p)f_i(p)^m\rp a_m(p,k)\mod p^N
\end{equation}
for every sufficiently large prime $p$.

Now, for $p$ larger than $\max\{f_i(1)\}$, we have
\[
a_0(p,k)=\begin{cases}1-\frac{1}{p}\text{ if $k=0$,}\\0\text{ otherise,}\end{cases}
\]
for every $k$ for which the inner sum in \eqref{eq1} is non-empty. This allows us to rewrite \eqref{eq1} using no terms $a_0(p,k)$. We can also eliminate the terms $a_1(p,k)$ for $k\geq 2$. By Proposition \ref{propser} we have
\[
B_k=(1-p^{k-1})^{-1}\sum_{m\geq 0}a_m(p,k)k^m,
\]
so we can solve for $a_1(p,k)$ in terms of $a_m(p,k)$, $m\geq 2$:
\[
a_1(p,k)=(1-p^{k-1})\frac{B_k}{k}-\sum_{m\geq 2}a_m(p,k) k^{m-1}.
\]
We substitute the expressions for $a_0(p,k)$ and $a_1(p,k)$ into \eqref{eq1} to obtain
\begin{gather*}
A(p)\equiv g_0(p)-\sum_{\substack{k\leq-2\text{ even}\\m\geq 1}}\sum_{\substack{i\\f_i(1)= k}} g_i(p)f_i(p)^m a_m(p,k)\\
-\lp 1-\frac{1}{p}\rp\sum_{\substack{i\\f_i(1)= 0}} g_i(p)-\sum_{m\geq 1}\sum_{\substack{i\\f_i(1)= 0}} g_i(p)f_i(p)^m a_m(p,0)\\
-\sum_{\substack{k\geq2\text{ even}}}\lp 1-p^{k-1}\rp\frac{ B_k}{k} \sum_{\substack{i\\f_i(1)= k}} g_i(p)f_i(p)\\
-\sum_{\substack{k\geq2\text{ even}\\m\geq 2}}\sum_{\substack{i\\f_i(1)= k}} g_i(p)\big(f_i(p)^m-k^{m-1}f_i(p)\big)a_m(p,k) \mod p^N.
\end{gather*}
Finally, by condition (1) of the hypothesis of the theorem,
\begin{gather*}
g_0(p)-\lp 1-\frac{1}{p}\rp\sum_{\substack{i\\f_i(1)= 0}} g_i(p)\\
-\sum_{\substack{k\geq2\text{ even}}}\lp 1-p^{k-1}\rp B_k \sum_{\substack{i\\f_i(1)= k}} g_i(p)f_i(p)\equiv 0\mod p^N.
\end{gather*}
By condition (2) of the hypothesis of the theorem, for every even $k\leq 0$ and $m\geq 1$,
\[
\sum_{\substack{i\\f_i(1)= k}} g_i(p)f_i(p)^m a_m(p,k)\equiv 0\mod p^N.
\]
By condition (3) of the hypothesis of the Theorem, for every even $k\geq 2$ and $m\geq 2$,
\[
\sum_{\substack{i\\f_i(1)= k}} g_i(p)\big(f_i(p)^m-k^{m-1}\big)a_m(p,k)\equiv 0\mod p^N.
\]
We conclude that $A(p)\equiv 0\mod p^N$ for all but finitely many $p$.
\end{proof}

\bibliographystyle{hplain}
\bibliography{jrbiblio}

\newcommand{\noop}[1]{} \def\cprime{$'$}
\begin{thebibliography}{1}

\bibitem{Coh07}
Henri Cohen.
\newblock {\em Number theory. {V}ol. {II}. {A}nalytic and modern tools}, volume
  240 of {\em Graduate Texts in Mathematics}.
\newblock Springer, New York, 2007.

\bibitem{Sun00}
Zhi-Hong Sun.
\newblock Congruences concerning {B}ernoulli numbers and {B}ernoulli
  polynomials.
\newblock {\em Discrete Appl. Math.}, 105(1-3):193--223, 2000.

\bibitem{Was82}
Lawrence~C. Washington.
\newblock {\em Introduction to cyclotomic fields}, volume~83 of {\em Graduate
  Texts in Mathematics}.
\newblock Springer-Verlag, New York, second edition, 1997.

\end{thebibliography}
\end{document}